\documentclass[12pt,twoside]{article}
\usepackage{amsfonts}
\usepackage{amssymb}
\usepackage{amsmath}
\usepackage{revsymb}
\usepackage{eurosym}
\usepackage{graphicx}
\usepackage{amsthm}
\usepackage[english]{babel}
\usepackage{cite}

\newtheorem{theorem}{Theorem}

\newtheorem{corollary}[theorem]{Corollary}

\newtheorem{lemma}[theorem]{Lemma}

\newtheorem{remark}[theorem]{Remark}

\title{\textbf{A Method to Estimate a Neighborhood of a Periodic Orbit}}
\author
{
\textbf{Mario Cavani} \\ \\
\small{Universidad de Oriente, Escuela de Ciencias, Cuman\'a, Venezuela.}\\
\small{\texttt{mcavani@udo.edu.ve}}}
\begin{document}
\maketitle
\begin{abstract}
In this paper we describe a method to estimate a neighborhood containing a periodic orbit of a given system of two ordinary differential equations. By using the theory of integral averages, the system of differential equations can be transformed into an equivalent autonomous system which, by using the Hopf bifurcation theorem, the existence of a periodic solution of this autonomous nonlinear differential equations can be demonstrated. Using of this procedure it is possible to estimate an annular region where the orbit of the periodic solution is located. The method allows to improve the results that the Hopf Bifurcation provides on periodic solutions. In addition, some quantitative characteristics of the solution can be known, such as the amplitude, the period and, a region where the periodic orbit of the original system is located. The method is applied to a three-dimensional system of differential equations that models the competition of two predators and one prey, which under the assumption that the predators are equally voracious, property that in this case leads to a two-dimensional system, where all of the conditions of the method described here are easily applicable.
\end{abstract}
\pagestyle{myheadings} \markboth{\centerline {\scriptsize M. Cavani
}}{\centerline {\scriptsize  A method to Estimate a Neighborhood of a Periodic Orbit}}


\section{Introduction}

This article deals with the existence of periodic solutions of nonlinear
ordinary differential equations, specifically it has to do with the periodic
solutions of autonomous systems (and can also be applied to periodic systems
concerning time, t) of nonlinear differential equations, which depend
on a parameter. This topic has occupied the attention of many mathematicians
since Poincar\'{e}, and in recent decades many advances have been made, but
its complete solution still seems to be a long way off. Many results,
generally difficult to obtain, apply to a very particular family of
equations (see \cite{CavaniAkhterAnote}), and in certain cases, to a single equation.
Among the results that have a certain degree of generality is the obtaining of periodic
solutions through the Hopf Bifurcation Theorem, and one of the objectives of the
present work is to describe this procedure in the case of two-dimensional
equations. Hopf's Bifurcation Theorem allows, in many cases, to solve the
the problem of the existence of periodic solutions in systems of differential
equations that depend on a parameter, but in practice, it may be necessary to
quantitatively know some characteristics of the solution, such as the
amplitude, the period or a region where it is located. In this direction,
there is the method of integral averages, which is also described in this
work in a general way, and is subsequently applied to problems related to
the Hopf Bifurcation. The study carried out in this article contributes to
the dissemination of a method that allows obtaining, approximately, the
periodic solutions of autonomous systems of ordinary differential equations
that depends on a parameter. The application of the methods presented here
poses some difficulty in calculations, because some steps depend on a
linearization, which is assumed to take the Jordan canonical form. The main
objective of this work is to show the step by step method to obtain approximately,
the amplitude and stability of a periodic solution obtained through the bifurcation of
a critical point of an autonomous equation.
\newline Here, it is analyzed a three-dimensional  system of differential equations that
models the competition of two predators and a prey that was proposed by
Hsu, Hubbell, and Waltman in \cite{HHW}. The results here,
improve the results of M. Farkas in \cite{MFarkas} because  gives an approximation
of the amplitude of the periodic solution, as well as an estimate of a annular region
that contains the orbit of the periodic solution.

\section{Preliminaries}
Let $a$ and $\alpha_0 \ge 0$ be real numbers, and consider the $n$-dimensional differential
equation
\begin{equation}
x^{\prime }=F(t,x,\alpha )  \label{eq1}
\end{equation}
here, $F:[a,+\infty )\times \mathbb{R} ^{n}\times [-\alpha _{0},\alpha
_{0}]\longrightarrow \mathbb{R} ^{n}$, is continuously differentiable with
respect to each one of its variables and is $T$-periodic in $t$, $T>0$.
\subsection{Periodic Solutions}
The following result characterizes the $T$-periodic solutions of system (\ref{eq1}).
\begin{lemma}
\label{lem1} A necessary and sufficient condition for a solution $x(t,\alpha
)$ of the equation (\ref{eq1}) to have period $T$ is,
\begin{equation}
x(T,\alpha )-x(0,\alpha )=0.  \label{cond}
\end{equation}
\end{lemma}
\begin{proof}
The condition is necessary. To see sufficiency, suppose that
$x(t,\alpha )$  is a solution of (\ref{eq1}) that satisfies (\ref{cond}). The
function,  $y(t)$, defined by $y(t)=x(t+T,\alpha )$ satisfy,
\begin{equation*}
y^{\prime }(t)=x^{\prime }(t+T,e)=F(t+T,x(t+T,\alpha ),\alpha
)=F(t,y(t),\alpha )\text{,}
\end{equation*}
and, $y(0)=x(T,\alpha )=x(0,\alpha )$.   From the Theorem of Existence and
the uniqueness of the solutions of a system of first-order differential
equations (see \cite{CodLevin} , page 1) is obtained that
$x(t+T,\alpha )=y(t)=x(t,\alpha ).$
\end{proof}
\begin{theorem}
\label{teo 2.2} Suppose that for $\alpha =0$, the system (\ref{eq1}) has a $%
T $-periodic solution, $p(t)$, such that the variational equation%
\begin{equation}
y^{\prime }=D_{x}F(t,p(t),0)y  \label{variaciones}
\end{equation}%
respect to this solution it has no $T$-periodic solutions, except the
trivial solution $y(t)=0$. Then for each value of \ $\alpha ,$ $|\alpha |$
small enough, the system (\ref{eq1}) has a unique $T$-periodic solution,
continuous in $(t,\alpha )$ that satisfies%
\begin{equation}
\lim_{\alpha \longrightarrow 0}x(t,\alpha )=p(t),  \label{lim-0}
\end{equation}%
Uniformly in the variable $t$.
\end{theorem}
\begin{proof}
Consider the solution of (\ref{eq1}) that for $t=0$ has the value $%
p(0)+c$, $\Vert c\Vert $ small enough, and denote this solution by%
\begin{equation*}
\beta =\beta (t,c,\alpha )
\end{equation*}
According to lemma (\ref{lem1}), the solution  $\beta $ will be $T$-periodic
if it holds that
\begin{equation}
\beta (T,c,\alpha )-p(0)-c=0.  \label{cond1}
\end{equation}
Therefore, the problem is reduced to finding pairs $(\alpha,c)$ of the
expression (\ref{cond}). To do this, the implicit function theorem is used.
Since for $\alpha =0$, we have $c=0$ as the solution of (\ref{cond1}),
therefore, if for this pair of values the Jacobian of the left part of (\ref{cond1})
is not zero, so for $|\alpha |$ small the relation (\ref{cond1})
has a unique solution $c=c(\alpha )$, such that,
\begin{equation*}
\lim_{\alpha \longrightarrow 0}c(\alpha )=0.
\end{equation*}
In this way, using the values obtained from the implicit function theorem,
the only $T$-periodic solution
\begin{equation*}
x(t,\alpha )=\beta (t,c(\alpha ),\alpha )
\end{equation*}
is obtained for the equation (\ref{eq1}). To determine the Jacobian in
question, the derivatives are taken for the components of $c$
and it is given by the matrix
\begin{equation}
D_{c}\beta (T,0,0)-I  \label{2.6}
\end{equation}
where $I$ is the identity matrix. This Jacobian is closely related to the
variational equation (\ref{variaciones}).
\begin{equation*}
\beta ^{\prime }(t,c,\alpha )=F(t,\beta (t,c,\alpha ),\alpha ),
\end{equation*}
and the derivative of the previous expression for $c$, and putting $\alpha =0$, $c=0$,  gives
\begin{equation*}
D_{c}\beta ^{\prime }(t,0,0)=D_{\alpha }F(t,\beta (t,0,0),0)D_{c}\beta
(t,0,0).
\end{equation*}
As $\beta (t,0,0)=p(t)$, then
\begin{equation*}
D_{c}\beta ^{\prime }(t,0,0)=D_{x}F(t,p(t),0)D_{c}\beta (t,0,0).
\end{equation*}%
From the hypotheses about $F$, the change in the order of derivation with
respect to $t$ and $c$, which is permissible, so it is concluded that the
matrix is given by
\begin{equation}
A(t)=D_{c}\beta (t,0,0)  \label{2.7}
\end{equation}%
is a matrix solution for the variational equation (\ref{variaciones}).
Now taking into account that
\begin{equation*}
\beta (0,c,\alpha )=p(0)+c,
\end{equation*}%
the derivative respect to $c$ in the previous expression gives
\begin{equation*}
D_{c}\beta (0,c,\alpha )=I.
\end{equation*}%
The characteristic multipliers are the eigenvalues of the matrix $A(T)$ given by
(\ref{2.7}), therefore the characteristic multipliers associated with
the system (\ref{variaciones}) are roots of the equation%
\begin{equation}
det(A(T)-\lambda I)=0  \label{2.8}
\end{equation}%
From the hypotheses about the system (\ref{variaciones}) and from the
general theory about systems of linear differential equations with periodic
coefficients, it is known that a system of this type has a non-trivial $T$%
-periodic solution, if and only if, $\lambda =1$ is a multiplier associated
to the system. In this case, $\lambda =1$ cannot be a multiplier since there
are no $T$-periodic solutions other than the null solution $y=0$ for all $t$. 
Which leads to the conclusion that, for $\lambda =1$ the expression (\ref%
{2.8}) is not zero. This expression coincides with the determinant of the
matrix given in (\ref{2.6}), which implies the existence of $T$-periodic
solutions for small $\Vert \alpha \Vert $. The limit stated in (\ref{lim-0})
is a direct consequence of the continuous dependence of the solutions for parameters.
\end{proof}

\begin{remark}
It should be noted that in the system (\ref{variaciones}) the condition of
not having another $T$-periodic solution other than the null solution is
equivalent to the same system not presenting characteristic exponents
integer multiples of $2\pi i/T$ (see \cite{CodLevin}, page 80). Therefore,
under this condition the conclusions of Theorem \ref{teo 2.2} remain valid.
\end{remark}

\begin{remark}
\label{obs 2.3} Theorem \ref{teo 2.2} is clearly applicable to the case in
which the right part of the equation (\ref{eq1}) is a small perturbation of
a linear system, that is, $F(t,x,\alpha )=A(t)x+\alpha f(t,x,\alpha )$,
where $f(t,x,\alpha )$ is of order $O(\Vert x\Vert ^{2})$, $\alpha $ is
small and the equation $y^{\prime }=A(t)y$ has an isolated periodic solution.
\end{remark}

Using the Theorem \ref{teo 2.2}, the study of the stability of the
periodic solutions of the equation (\ref{eq1}) can be realized. Let be,
\begin{equation}
x=z+p(t)  \label{2.9}
\end{equation}%
with, $x=x(t,\alpha )$, $z=z(t,\alpha )$, and $p(t)$ the $T$-periodic
solution of the equation (\ref{eq1}) obtained in the Theorem \ref{teo 2.2},
then the equation (\ref{eq1}) can be rewritten as
\begin{equation}
z^{\prime }=D_{x}F(t,p(t),0)z+f(t,z,\alpha ),  \label{2.10}
\end{equation}%
where, $f(t,z,\alpha )=O(\Vert z\Vert ),$ $\Vert z\Vert $ small enough,
uniformly for $t\in \lbrack a,a+T],\quad \alpha \in \left[ -\alpha
_{0},\alpha_{0}\right] .$ From the equation (\ref{2.10}) it can be seen
that as an immediate consequence of Perron's Theorem (see \cite{CodLevin},
Chapter 13, Theorem 1.1 and Theorem 1.4), using the change of variables
established in the well-known Floquet Theory (see also \cite{Hale80}, page
118) that allows passing from a differential equation of periodic
coefficients to one of the constant coefficients, it can be established that the
asymptotic stability of the solution $p(t)$ follows from the asymptotic
stability of the zero solution of the system

\begin{equation}
y^{\prime }=D_{x}F(t,p(t),0)y.  \label{2.11}
\end{equation}%
Indeed the following theorem holds:

\begin{theorem}
\label{teo exp} If the characteristic exponents associated with the equation
in variations (\ref{2.11}) all have a negative real part, then the
The $T$-periodic solution, $p(t)$ of the equation (\ref{eq1}) is asymptotically
stable.
\end{theorem}

The asymptotic stability of the $T$-periodic solutions is given by the Theorem
\ref{teo 2.2} \ can be established as follows:

\begin{theorem}
If the equation in variations (\ref{2.11}) of the system (\ref{eq1}) with
respect to the $T$-periodic solution $p(t)$ that is obtained for $\alpha =0$
is such that the real parts of all its characteristic exponents are
negative, then the $T$-periodic solution $x(t,\alpha )$ obtained in Theorem %
\ref{teo 2.2} for $|\alpha |$ small is asymptotically stable.
\end{theorem}

\begin{proof}
It should be noted that, according to the hypotheses stated in
Observation  \ref{obs 2.3}, the conclusions of Theorem \ref{teo 2.2}
continue to be obtained. Now, the equation in variations of the equation (%
\ref{eq1}) with respect to the $T$-periodic solution $x(t,\alpha )$ is
given by the system of periodic coefficients
\begin{equation}
y^{\prime }=D_{x}F(t,x(t,\alpha ),\alpha )y.  \label{2.12}
\end{equation}%
It is observed that the equation (\ref{2.12}) reduces to the equation (\ref%
{2.11}) for $\alpha =0$. According to this, if \ $B=B(t,\alpha )$ is a
fundamental matrix of the equation (\ref{2.11}), such that $B(0,\alpha )=I$%
, $I$ the identity matrix, then the characteristic multipliers associated
with the equation (\ref{2.12}) are the eigenvalues of the matrix $B(T,%
\alpha )$. Taking into account the continuity of $B$ with respect to $%
\alpha $, and the fact that the eigenvalues of $B(T,0)$ have a modulus
less than one, then so will the eigenvalues of $B(T,\alpha )$ with $|%
\alpha |$ small enough. For each such $\alpha $, a translation of the fixed
$\alpha $ to zero can be carried out, subsequently applying the Theorem \ref{teo exp}
with $x(t,\alpha )$ instead of $p(t)$, and so on its asymptotic
stability is obtained.
\end{proof}

\subsection{The Hopf Bifurcation Theorem in ${\mathbb{R} }^{2}$}
This section presents the Hopf Bifurcation Theorem for systems of
differential equations in two dimensions, in this theorem sufficient
conditions are given for the existence of a unique periodic solution located
in a neighborhood of an equilibrium solution. The outline of the
the presentation follows that made in \cite{NegSal}. Consider a
mono-parametric family of ordinary differential equations defined by
\begin{equation}
x^{\prime }=f(x,\alpha )  \label{3.1}
\end{equation}%
$f$ $\in $ $C^{k+1}$, the class of the functions with $(k+1)$-continuous
derivatives, $k\geq 3$, $f(0,\alpha )=0,$ for all $\alpha .$ Suppose that $%
x=0$ is an equilibrium point for the equation (\ref{3.1}), and has the
the following linearization
\begin{equation}
y^{\prime }=A(\alpha )y,  \label{3.2}
\end{equation}%
such that for $\alpha <0$, $x=0$ it is asymptotically stable, a center when $%
\alpha =0$ and unstable for $\alpha >0$. Denote by $\beta (\alpha )\pm i \gamma (\alpha )$
the eigenvalues corresponding to the matrix in the right
side of the equation (\ref{3.2}) and in relation to them suppose that
\begin{equation}
\beta (0)=0,\beta ^{\prime }(0)>0,\quad \gamma (0)>0.  \label{3.3}
\end{equation}%
The equation (\ref{3.1}) can be represented in the form
\begin{equation}
x^{\prime }=A(\alpha )x+F(x,\alpha ),\quad \left( \Vert F(x,\alpha )\Vert
\right)=O\left( \Vert x\Vert ^{2}\right) .  \label{3.4}
\end{equation}%
Without loss of generality, suppose that the matrix $A(\alpha )$ has the
Jordan canonical form, with the following notations:%
\begin{equation*}
x=col\left( x_{1},x_{2}\right),\quad F=col(p,q),
\end{equation*}%
\begin{eqnarray}
x_{1}^{\prime } &=\beta (\alpha )x_{1}-\gamma (\alpha )x_{2}+p\left(
x_{1},x_{2},\alpha \right)  \notag \\
&&  \label{3.5} \\
x_{2}^{\prime } &=\gamma (\alpha )x_{1}+\beta (\alpha )x_{2}+q\left(
x_{1},x_{2},\alpha \right) .  \notag
\end{eqnarray}
From the conditions imposed on the function $f$ \ it follows that if in (\ref{3.5})
is made the change of variables,
\begin{equation*}
x_{1}=r\cos (\theta ),\;x_{2}=r\sin(\theta ),
\end{equation*}%
then
\begin{eqnarray}
r^{\prime } &=&\beta (\alpha )r+p^{\ast }(r,\theta ,\alpha )\cos (\theta
)+q^{\ast }(r,\theta ,\alpha )\sin(\theta )  \notag \\
r\theta ^{\prime } &=&r(\alpha )r+q^{\ast }(r,\theta ,\alpha )\cos (\theta
)-p^{\ast }(r,\theta ,\alpha )\sin(\theta ),  \notag
\end{eqnarray}
where,
\begin{equation*}
\begin{array}{l}
p^{\ast }(r,\theta ,\alpha )=p(r\cos (\theta ),r\sin(\theta ),\alpha ),
\\
q^{\ast }(r,\theta ,\alpha )=q(r\cos (\theta ),r\sin(\theta ),\alpha ).%
\end{array}%
\end{equation*}%
Now, the following function:
\begin{equation*}
G(r,\theta ,\alpha )=\gamma (\alpha )+\frac{q^{\ast }(r,\theta ,\alpha )\cos
(\theta )-p^{\ast }(r,\theta ,\alpha )\sin(\theta )}{r},\;r\neq 0,
\end{equation*}%
satisfy
\begin{equation*}
G(0,\theta ,\alpha )=\gamma (\alpha ).
\end{equation*}%
From the hypotheses about $p$ and $q$ it follows that the function $G$ is of
class $C^{k}$, and from the fact that $\gamma (0)>0$, it can be inferred
that there are positive real numbers $\alpha _{0}$ y $h$ such that $%
G(r,\theta ,\alpha )>0$, for $\alpha \in \left( -\alpha _{0},\alpha
_{0}\right) ,$ $r\in (0,h)$, and \ $\theta \in \mathbb{R} .$ For each$%
r_{0}\in \lbrack 0,h)$ y $\theta _{0}\in \mathbb{R} $ the orbit of (\ref{3.5}%
) with initial point $(r_{0},\theta _{0})$ is denoted by means of the no
continuable solution, $r\left( \theta ,r_{0},\theta _{0},\alpha \right) ,$
of the problem:
\begin{equation}
\frac{dr}{d\theta }=R(r,\theta ,\alpha ),\quad r\left( \theta
_{0},r_{0},\theta _{0},\alpha \right) =r_{0},  \label{3.6}
\end{equation}%
where,
\begin{equation}
R(r,\theta ,\alpha )=\frac{\gamma (\alpha )+p^{\ast }(r,\theta ,\alpha )\cos
(\theta )+q^{\ast }(r,\theta ,\alpha )\sin(\theta )}{G(r,\theta
,\alpha )}.  \label{3.7}
\end{equation}%
Once determined the functions $r(\theta ,r_{0},\theta _{0},\alpha )$, the
solutions of equation (\ref{3.5}) can be obtained integrating the following
equation
\begin{equation}
\theta ^{\prime }=G\left( r\left( \theta ,r_{0},\theta _{0},\alpha \right)
,\theta ,\alpha \right)  \label{3.8}
\end{equation}%
$\left( r_{0},\theta _{0},\alpha \right) $ $\in [O,h)\times \mathbb{R}
\times (-\alpha _{0},\alpha _{0})$. On the other hand, since $\gamma (0)>0$,
from (\ref{3.7}) it can be seen that if $h$ and $\alpha _{0}$ are selected
small enough, the right side of the equation (\ref{3.6}) can be limited in
such a way that they are defined in the entire interval $[0,2\pi ]$ (this
statement is obtained as a consequence of Picard's classical theorem on the
existence of solutions to the equations ordinary differentials, (\cite{CodLevin}, page 12).
For each $\alpha $\ the solution obtained as described
above is denoted by $r(\theta ,c,\alpha )$. Now consider the so-called
displacement function
\begin{equation*}
V(c,\alpha )=r(2\pi ,c,\alpha )-c
\end{equation*}%
associated to (\ref{3.5}). Using the fact that the function $R$ belongs to
the space $C^{k}$, then writing
\begin{equation}
r(\theta ,c,\alpha )=u_{1}(\theta ,\alpha )c+u_{2}(\theta ,\alpha
)c^{2}+\ldots +u_{k}(\theta ,\alpha )c^{k}+B(\theta ,c,\alpha ),  \label{3.9}
\end{equation}%
with $B$ is a function of order bigger than $k$ in the variable $c$, and the
condition  $r(0,c,\alpha )=c$ implies that
\begin{equation}
u_{1}(0,\alpha )=1,\;u_{2}(0,\alpha )=\ldots =u_{k}(0,\alpha )=B(0,c,\alpha
)=0.  \label{3.10}
\end{equation}%
Introducing the expression (\ref{3.9}) in the equation (\ref{3.6}) and
equating the corresponding coefficients with the same power of $c$, and
integrating in the following equations:

\begin{equation}
\frac{\partial u_{1}}{\partial \theta }=\frac{\beta (\alpha )}{r(\alpha )}%
u_{1}(\theta ,\alpha ),\quad u_{1}(0,\alpha )=1  \label{3.11}
\end{equation}%
\begin{equation}
\frac{\partial u_{l}}{\partial \theta }=U_{l}(\theta ,\alpha
),\;u_{l}(0,\alpha )=0,\;l=2,3,\ldots   \label{3.12}
\end{equation}%
it is possible to get the functions $u_{i}$ which are constructed step by
step from $u_{l}$, for $i<l.$ From the equation (\ref{2.11}) is obtained
that
\begin{equation*}
u_{1}(\theta ,\alpha )=\exp \left( \theta \frac{\beta (\alpha )}{\gamma
(\alpha )}\right).
\end{equation*}%
From the theory given previously, we can show how the $\mathbb{R}
^{2}$ version in $C^{k+1}$ of the Hopf Bifurcation Theorem.
\begin{theorem}
{Hopf Bifurcation} \label{Hopf} \newline
Considering the system (\ref{3.5}) and suppose that (\ref{3.3}) hold. Then
there exists a number $\delta$ in $(0,h)$ and a real function $\alpha $ from $(0,\delta )$ a belonging to $C^{k-1}$,
that satisfies
\begin{equation*}
\alpha (0)=\alpha ^{\prime }(0)=0.
\end{equation*}%
Also, the following holds, the number $\sigma $ defined by
\begin{equation*}
\sigma \equiv \sup \{|\alpha (c)|:c\in \lbrack 0,\delta )\}
\end{equation*}%
there exists and is lower than que $\alpha _{0}$, and with the property that
the closed periodic orbit of the system (\ref{3.5}) with $\alpha \in (-\sigma,\sigma )$ is
obtained only for $\alpha =\alpha (c)$.
\end{theorem}
\begin{proof}
The proof of this theorem will be obtained as a consequence of the
implicit function theorem. It can be observed that the orbit that passes
through the point $(c,0)$, $c\in (0,h)$, is closed if and only if $V(c,\alpha )=0$. Now, this equality is fulfilled for $c=0$ or also when
\begin{equation}
V(c,\alpha )=u_{2}(2\pi ,\alpha )-1+u_{2}(2\pi ,\alpha )c+\ldots +u_{k}(2\pi,\alpha )c^{k-1}+B(2\pi ,c,\alpha )=0. \label{3.13}
\end{equation}
According to the hypotheses $B(2\pi,c,\alpha )$ is of class $C^{k-1}$ and of
order greater than $k- 1$ concerning $c$. The function $u_{1}$ is known explicitly and $u_{1}(2\pi ,0)=\exp (0)=1,$ therefore the equation (\ref{3.13}%
) is satisfied for $c=\alpha =0$. Also from (\ref{3.13}) we obtain that
\begin{equation}
\frac{\partial V}{\partial \alpha }(0,0)=\frac{2\pi }{\gamma (0)}\beta
^{\prime }(0)=0, \label{3.14}
\end{equation}
and as can be seen, the hypotheses of the implicit function theorem
are fulfilled, from which the existence of the positive number $\delta $ and the function $\alpha :(0,\delta )\longrightarrow \mathbf{R},$
such that $\alpha (0)=0$ and $V(c,\alpha (c))=0$. When taking derivative with respect to $c$ in the identity $V(c,\alpha (c))=0$, we obtain:
\begin{equation}
\frac{\partial u_{1}}{\partial \alpha }(2\pi ,0)\alpha ^{\prime
}(0)+u_{2}(2\pi ,0)=0 \label{3.15}
\end{equation}
The equation (\ref{3.12}) with $l=2,$ has the form
\begin{equation*}
\frac{\partial u_{2}}{\partial \theta }=\frac{\beta (\alpha )}{\gamma
(\alpha )}u_{2}+\exp \left( \theta \frac{\beta (\alpha )}{\gamma (\alpha )}%
\right) (a_{0}\cos(\theta )+b_{0}\sin(\theta )),
\end{equation*}
by solving this first-order linear equation concerning $\theta $,
it can be easily observed that $u_{2}(2\pi ,0)=0,$ and from (\ref{3.15}) it is
obtained that $\alpha ^{\ast }(0)=0,$ since $\frac{\partial u_{1}}{\partial
\theta }(2\pi,0)$ is not zero. The condition $\alpha ^{\prime }(0)=0,$
guarantees that the number $\delta $ can be taken such that $\sigma
$ is less than $\alpha _{0}$, since otherwise $\alpha (c)\geq 0$
for all $\delta $ would lead to $\alpha ^{\prime }(0)\geq 1$, and this is
impossible. Which completes the proof
\end{proof}
An interesting model where was applied the Hopf bifurcation theorem is presented in
\cite{AkhterBitycoin}.
\section{The Integral Average Method}
This section theoretically describes the integral averaging procedure, which
in this work for convenience is called the averaging method. This analysis
technique is due to the Soviet physic-mathematicians N. Krilov and N.
Bogoliubov. A detailed exposition of the original approach to the theory is
found in\cite{BogMit}. A more recent presentation of this method can be seen in
\cite{Mustapha}.
The approach presented in the present work continues from that carried out by
J. Hale in \cite{Hale80}.
The average method allows, using an adequate change of variables, to take
a non-autonomous differential equation, which has a periodic solution, into
another equation that is autonomous and with the property that the mentioned
periodic solution can be approximated satisfactorily and with the advantage
Of obtaining information about its stability. The method is applied to
systems of equations that have the form:
\begin{equation}
x^{\prime }=\alpha f(t,x,\alpha )  \label{4.1}
\end{equation}
$t\in \mathbf{\mathbb{R}},$ $x\in U,$ $U$ a bounded subset of $\mathbf{\mathbb{R}}^{n}$,
with $\alpha$ such that $0<\alpha <1$ is a small parameter. Also is supposed that
$f:\mathbf{\mathbb{R}}\times \mathbb{R}^{n}\times\mathbb{R}^{+}\longrightarrow \mathbb{R}^{n}$
in the space of function $C^{k},k\geq 2$, is a bounded function on bounded sets in $\mathbb{R}^{n}$
with positive period $T$, respect to $t$.\newline
To describe the method consider beside system (\ref{4.1}) the following \textit{average system}:
\begin{equation}
x^{\prime }=\alpha f_{0}(x)  \label{4.2}
\end{equation}
where,
\begin{equation}
f_{0}(x)=\frac{1}{T}\int_{0}^{T}f(t,x,0)dt.  \label{4.3}
\end{equation}
The basic problem in using the average method consists of looking for in which sense
the solutions of the autonomous system (\ref{4.2}) are close to the solutions of the
more complicated and not autonomous system (\ref{4.1}). In the next theorem
it is shown that there exists a variable change that takes system (\ref{4.1}) in the
system
\begin{equation}
y^{\prime }=\alpha \left[ f_{0}(y)+\alpha f_{1}(t,y,\alpha )\right] +O\left(
\alpha ^{3}\right) ,  \label{4.4}
\end{equation}
and from this new system, as will be seen, it is possible to obtain the required information about the $T$-periodic solutions of the system (\ref{4.1}).
In the following result, all the necessary tools are given to
achieve the stated objective.
\begin{theorem}
{Integral Averages}\label{Averages}\newline
For the differential equation (\ref{4.1}) the following change of variable of class $C^{k}$
\begin{equation}
x=y+\alpha u(t,y), \label{4.5}
\end{equation}
where the function
$u$ is $T$-periodic with respect to the variable $t$, produces the following
results:
\begin{enumerate}
\item {Under the change of variables (\ref{4.5}) the differential equation (\ref{4.1}) takes the form (\ref{4.4}), such that the
function $f_{1}$ becomes $T$-periodic with respect to the variable $t$.}
\item {If $x(t)$ and $y(t)$ are the solutions of (\ref{4.1}) such that for $t=0$ take the values$x_{0}$ and $y_{0}$ respectively, and furthermore
$\left\vert x_{0}-y_{0}\right\vert =O(\alpha ),$ then $|x(t)-y(t)|=O(\alpha )$ on a time scale of order $O(\frac{1}{\alpha }$)}.
\item {If $y^{\ast }$ is a hyperbolic equilibrium point of equation (\ref{4.2}) then there exists a $\alpha _{0}>0$ such that, for all $\alpha$
in the interval $(0,\alpha_{0}]$, the differential equation (\ref{4.1}) has a
unique $T$-periodic solution $x(t,\alpha )$ in a neighborhood of $y^{\ast }$ continuous at $t$
and $\alpha $, and such that
\begin{equation*}
\lim_{\alpha \longrightarrow 0^{+}}x(t,\alpha )=y^{\ast }.
\end{equation*}
Furthermore, the solution $x(t,\alpha)$ has the same type of stability as the equilibrium
solution $y^{\ast }$ of equation (\ref{4.2})}.
\end{enumerate}
\end{theorem}
\begin{proof}
First, the calculations leading to the desired change of variables will be done,
\begin{equation*}
F(t,x,\alpha )=f(t,x,\alpha )-f_{0}(x)
\end{equation*}
and let
\begin{eqnarray*}
u(t,x) &=&\int_{0}^{t}{f(s,x,0)}-f_{0}(x)ds \\
&=&\int_{0}^{t}F(s,x,0)ds.
\end{eqnarray*}
The function $u$ is of class $C^{k}$ in all its variables. Neither
Now find out what is $T$-periodic in $t$:
\begin{eqnarray*}
u(t+T,x) &=&\int_{0}^{t+T}F(s,x,0)ds \\
&=&\int_{0}^{t}F(s,x,0)ds+\int_{t}^{t+T}F(s,x,0)ds \\
&=&u(t,x)+\int_{t}^{t+T}f(s,x,0)ds-\int_{t}^{t+T}f_{0}(x)ds \\
&=&u(t,x)+\int_{0}^{T}f(s,x,0)ds-\int_{0}^{T}f_{0}(x)ds \\
&=&u(t,x),
\end{eqnarray*}
With the function $u$ defined above, the change of variables $x=y+\alpha u(t,y)$ and it is observed that for small $\alpha$,
\begin{equation}
y=x-\alpha u(t,x)+O\left( \alpha ^{2}\right) , \label{4.6}
\end{equation}
given the invertibility of change. It can be observed that
\begin{eqnarray*}
\left\vert \frac{y-x+\alpha u(t,x)}{\alpha ^{2}}\right\vert &=&\left\vert
\frac{u(t,x)-u(t,y)\mid }{\alpha }\right\vert \\
&=&\left\vert \frac{u(t,y+\alpha u(t,y))-u(t,y)}{\alpha }\right\vert \\
&=&\left\vert \frac{u(t,y+\alpha \mid u(t,y))-u(t,y)+O(\alpha )}{\alpha }%
\right\vert \\
&\leq &\left\vert \frac{u(t,y+\alpha u(t,y)-u(t,y)}{\alpha }\right\vert
+\left\vert \frac{O(\alpha )}{\alpha }\right\vert
\end{eqnarray*}%
which confirms (\ref{4.6}).\newline
Now, if we take the derivative with respect to $t$ in the expression (\ref{4.5}) we obtain,
\begin{eqnarray*}
x^{\prime } &=&y^{\prime }+\alpha u^{\prime }(t,y)+\alpha
D_{y}u(t,y)y^{\prime } \\
&=&\left[ I+\alpha D_{y}u(t,y)\right] y^{\prime }+\alpha u^{\prime }(t,y),
\end{eqnarray*}%
from which,
\begin{equation*}
\begin{aligned} \left[I + \alpha D_y u(t,y)\right] y^{\prime} &= x^{\prime}
- \alpha u^{\prime}(t,y) \\ & = \alpha f(t,x,\alpha) - \alpha
u^{\prime}(t,y) \\ & = \alpha f(t,y + \alpha u(t,y),\alpha) - \alpha
u^{\prime}(t,y) \\ & = \alpha \left(f(t,y,\alpha)+\alpha D_y f(t,y,\alpha)
u(t,y) +O(\alpha^2) - u^{\prime}(t,y)\right) \\ & = \alpha
(f_0(y)+F(t,y,0)+\alpha D_y f(t,y,\alpha) u^{\prime}(t,y)\\ &\quad+ O
(\alpha^2) - u^{\prime}(t,y). \end{aligned}
\end{equation*}%
But, remembering that $u^{\prime }(t,y)=F(t,y,0)$ and taking into account the
invertibility of the change of variables given in (\ref{4.5}), we have that
\begin{eqnarray*}
y^{\prime } &=&\left[ I+\alpha D_{y}u(t,y)\right] ^{-1}f_{0}(y)+\alpha
D_{y}f(t,y,\alpha )u(t,y)+O(\alpha ^{2}) \\
&=&\alpha f_{0}(y)+\alpha ^{2}D_{y}f(t,y,\alpha )u(t,y)-\alpha
^{2}D_{y}u(t,y)f_{0}(y)+O(\alpha ^{3}),
\end{eqnarray*}%
which leads to the equation
\begin{equation*}
y^{\prime }=\alpha f_{0}(y)+\alpha ^{2}f_{1}(t,y,\alpha )+O(\alpha ^{3}),
\end{equation*}%
where, as can be seen, the function $f_{1}$ is $T$-periodic at t. With this,
part 1 of the theorem has been proved.\newline
To prove part 2, let $y(t)$ and $y_{\alpha }(t)$ be the
respective solutions of equations (\ref{4.2}) with initial condition $y_{0}$ and (\ref{4.4}) with initial condition $y_{\alpha _{0}}$, both
at $t=0$. These solutions satisfy:
\begin{equation*}
\begin{aligned} y_{\alpha}(t)-y(t)=& y_{\alpha_0} -y_{0}+\alpha
\int_{0}^{t}\left[f_{0}\left(y_{\alpha}(s)\right)-f_{0}(y(s))\right] d s \\
&-\alpha^{2} \int_{0}^{t} f_{1}\left(y_{\alpha}(s), s, \alpha\right) d s .
\end{aligned}
\end{equation*}%
Let $c(t)=\left( y_{\alpha }-y\right) (t),$ then
\begin{equation*}
|c(t)|\leq |c(0)|+\alpha L\int_{0}^{t}|c(s)|ds+\alpha ^{2}Mt,
\end{equation*}%
where, $L$ is the Lipschitz constant of $f_{0}$ and $M$ is the supremum of $%
f_{1}$ on the set U. Now applying Gronwall's Lemma (\cite{Hale80}) we have
\begin{eqnarray*}
|c(t)| &\leq &|c(0)|e^{Lt}+\alpha ^{2}M\int_{0}^{t}e^{L(t-b)}ds \\
&\leq &(|c(0)|+\alpha M/L)e^{Lt}.
\end{eqnarray*}%
From this we conclude that, if $\left\vert y_{\alpha _{0}}-y_{0}\right\vert
=O(\alpha )$, then for all $t\in \lbrack 0,\frac{1}{\alpha L}]$,
\begin{equation*}
\left\vert y_{\alpha }(t)-y(t)\right\vert =O(\alpha ).
\end{equation*}%
It can also be observed, according to the transformation (\ref{4.5}), that
\begin{equation}
\left\vert x(t)-y_{\alpha }(t)\right\vert =\alpha \left\vert u\left(
y_{\alpha },t,\alpha \right) \right\vert =O(\alpha ) \label{4.7}
\end{equation}
and finally, using the triangular inequality
\begin{equation*}
|x(t)-y(t)|\leq \left\vert x(t)-y_{\alpha }(t)\right\vert +\left\vert
y_{\alpha }(t)-y(t)\right\vert ,
\end{equation*}
from which it is observed that $|x(t)-y(t)|=O(\alpha ),$ which completes the
proof of part 2; However, it should be noted that by (\ref%
{4.7}), the hypothesis $\left\vert x_{0}-y_{0}\right\vert =O(\alpha )$ and the
triangular inequality gives that $\left\vert y_{\alpha_{0}}-y_{0}\right\vert =O(\alpha ),$
an issue that has been used fundamentally in the proof.\newline
To prove part 3, consider that $y^{\ast }$ is a hyperbolic equilibrium point of (\ref{4.2}), to address the problem of the
existence of $T$-periodic solutions for equation (\ref{4.1}) it must be
remembered that, according to Lemma \ref{lem1} the relation must be satisfied n
\begin{equation}
x\left( 0,x_{0},\alpha \right) =x\left( T,x_{0},\alpha \right) , \label{4.8}
\end{equation}%
using the fact that the solutions of (\ref{4.1}) also satisfy the
equality
\begin{equation*}
x\left( t,x_{0},\alpha \right) =x_{0}+\alpha \int_{0}^{t}f\left( s,x\left(
s,x_{0},\alpha \right) ,\alpha \right) ds,
\end{equation*}%
(\ref{4.8}) is equivalent to
\begin{equation*}
\int_{0}^{T}f\left( s,x\left( s,x_{0},\alpha \right) ,\alpha \right) ds=0.
\end{equation*}%
Putting,
\begin{equation*}
G(x_{0},\alpha )=\int_{0}^{T}f\left( s,x\left( s,x_{0},\alpha \right)
,\alpha \right) ds=0,
\end{equation*}%
taking into account the hypotheses $f_{0}(y^{\ast })=0$ and that the derivative $%
D_{x}f_{0}(y^{\ast })$ is invertible, the calculations are as follows:
\begin{eqnarray*}
G(y^{\ast },0) &=&\int_{0}^{T}f\left( s,x\left( s,y^{\ast },0\right)
,0\right) ds \\
&=&\int_{0}^{T}f(s,y^{\ast },0)ds=Tf_{0}(y^{\ast })=0,
\end{eqnarray*}%
and furthermore,
\begin{eqnarray*}
D_{x_{0}}G(y^{\ast },0) &=&\int_{0}^{T}D_{x}f\left( s,x\left( s,x_{0},\alpha
\right) ,\alpha \right) D_{x_{0}}x(s,x_{0},\alpha )ds|x_{0}=y^{\ast
},\;\alpha =0 \\
&=&\int_{0}^{T}f(s,y^{\ast },0)D_{x_{0}}x(s,y^{\ast },0)ds=Tf_{0}(y^{\ast
})=0.
\end{eqnarray*}
It is known, see \cite{Hale80}, that $D_{x{0}}x(s,y^{\ast },0)$ satisfy the variational equation
\begin{equation*}
\frac{dw}{ds}=\alpha D_{x}f(s,y^{\ast },\alpha )w,\quad w(0,y^{\ast },\alpha
)=I,
\end{equation*}
therefore, for $\alpha =0$ it is clear that $D_{x_{0}}x(s,y^{\ast
},0)=I $ \ $s\geq 0$. So, it is concluded that
\begin{equation*}
D_{x_{0}}G(y^{\ast },0)=\int_{0}^{T}D_{x}f(s,x(s,y^{\ast
},0)ds=TD_{x}f_{0}(y^{\ast }),
\end{equation*}
which implies the invertibility.
According to the Implicit Function Theorem, there exists $\alpha _{0}>0,$ a neighborhood centered at $y^{\ast }$ and radius $h$,
$N\left( y^{\ast },h\right) ,\;h>0$ and a function $x_{0}:\left( -\alpha
_{0},\alpha _{0}\right) \longrightarrow N\left( y^{\ast },h\right),$ such
that $G\left( x_{0}(\alpha ),\alpha \right) =0$ for all $\alpha $ and the
solution of the problem
\begin{equation*}
x^{\prime }=\alpha f(t,x,\alpha ),x(0)=x_{0}(\alpha )
\end{equation*}%
is $T$-periodic, and it is also the only $T$-periodic solution in $N\left( y^{\ast
},h\right) $ for each $\alpha \in (-\alpha _{0},\alpha _{0})$. From the Theorem
From the Implicit Function quote, it is also clear that the periodic solution in
question is continuous in $(t,\alpha )$ and $\lim_{\alpha \longrightarrow
0^{+}}x(t,\alpha )=y^{\ast },$ where $x(t,\alpha )$ denotes the $T$-periodic
solution, for each $\alpha $. To summarize, note that using the
change of variables (\ref{4.5}) the equation (\ref{4.1}) is brought to the form
(\ref{4.4}) and in this last equation the change of variables $y=z+y^{\ast }$
is performed to obtain the equation
\begin{equation}
z^{\prime }=\alpha Az+\alpha \left\{ \alpha f_{1}\left( t,z+y^{\ast },\alpha
\right) +f_{0}\left( z+y^{\ast }\right) -f_{0}\left( y^{\ast }\right)
-Az\right\} \label{4.9}
\end{equation}
where, $A=D_{x}f_{0}\left( y^{\ast }\right) $. The equation (\ref{4.9}) can
be considered as a perturbation of the linear system
\begin{equation}
z^{\prime }=\alpha Az. \label{4.10}
\end{equation}%
If the zero solution of (\ref{4.10}) is asymptotically stable, it will be
the $T$-periodic solution of (\ref{4.9}) corresponding to each $\alpha \in
(-\alpha _{0},\alpha _{0})$. If any eigenvalue of $A$ has a positive real part,
the zero solution of (\ref{4.10}) is unstable, therefore the
instability will also correspond to the corresponding periodic solution
of (\ref{4.9}) for each $\alpha \in (-\alpha _{0},\alpha _{0})$. Since all
the changes of variables that have been made are invertible, it follows that
the $T$-periodic solutions found for equation (\ref{4.1})
inherit the same type of stability as the hyperbolic equilibrium point
of (\ref{4.2}). Which completes the proof.
\end{proof}

\section{Method to Find the Neighborhood}
In the previous sections, two methods have been studied that allow
establishing the existence of $T$-periodic solutions of systems of
Differential equations. In Section 3, we saw the sufficient conditions under
which is a two-dimensional differential equation of autonomous type, which
depends on a parameter, has $T$-periodic solutions. The key result is that
the case was Hopf's Bifurcation Theorem. However, no results concerning the
stability of the orbit of the periodic solution was discussed there.
However, in that direction, some theorems guarantee stability, for
example in \cite{MMacCrack} gives results in this regard,
nevertheless, the techniques used therein
practice is difficult to apply. On the other hand, in Section 4, the
averaging method that is applied to non-autonomous equations is studied,
more precisely, the theory presented has its natural environment in
differential equations with right part $T$-periodic and continuously
differentiable up to a certain order $k\geq 2$, in such a way that the
calculations to be carried out are possible. There it was also seen that, in
the case of the averaged system having a hyperbolic equilibrium point, the
$T$-periodic solution of the originally proposed system presents the same
type of stability as such an equilibrium point. In this Section, it is proposed
to establish a connection between the two methods. The situation would be
like this: given a two-dimensional autonomous differential equation
satisfying the hypotheses of the Hopf Bifurcation Theorem, it is introduced
into the equation an independent variable using polar coordinates, then the
average of the resulting equation is calculated in such a way that the periodic
solution obtained from the average method, for uniqueness, must be
the same as that corresponding to the found by Hopf's Theorem and now with
the possibility of being able to decide on the stability of the $T$-periodic
solution and to be able to estimate approximately. To carry
out this task, some fundamental results from the article by S.N. Chow and J.
Mallet-Paret in \cite{ChowMallet}. The techniques developed in this article
have a notable influence on the method presented in this section. The method developed
here can be generalized to systems with dimensions greater than two, for
this it is interesting to take into account the theory presented in the book
of Kuznetsov \cite{Kuznetsov}.
But one can always try to make a finite expansion of $r_{1}$ in powers of $\alpha$ until achieving a fairly accurate approximation to $r$. What follows is based on this last idea and the results in \cite{ChowMallet} are based on this. Under this vision, the extended average method is now described, in the sense cited, to obtain the periodic solution resulting from the Hopf Bifurcation Theorem expanded up to order $k$ concerning a parameter $\mu $ in two-dimensional equations. Consider the equations in polar coordinates
\begin{equation}
\begin{aligned} r^{\prime}&=\mu R_{1}(r, \theta, \alpha)+\mu^{2} R_{z}(r,
\theta, \alpha)+\ldots \\ \theta^{\prime}&=w+\mu W_{1}(r, \theta,
\alpha)+\mu^{2} W_{2}(r, \theta, \alpha)+\ldots, \end{aligned}  \label{5.12}
\end{equation}
$\mu,\alpha $ are parameters, $\mu \in \left( -\mu_{0},\mu_{0}\right) $
and $w$ is constant,
the functions on the right-hand side of (\ref{5.12}) are assumed to be of class $C^{k}, $ for $k$ suitable for the calculations to be performed. Furthermore, only a finite number of coefficients $\mu_{k}$ are considered, and it is therefore sufficient that (\ref{5.12}) is represented by a finite Taylor series with remainder. In bifurcation problems, $\alpha$ represents the parameter that, when varied, produces the bifurcation and $\mu$ is a conversion factor to a new scale.\newline
It is clear that, if all functions $R_{j}$ are independent of $\theta $ then the circles $r=r_{0},$ where,
\begin{equation*}
\mu R_{1}\left( r_{0},\alpha \right) +\mu ^{2}R_{2}\left( r_{0},\alpha
\right) +\ldots =0
\end{equation*}
are periodic solutions of (\ref{5.12}). It follows from this that to
find periodic solutions of this equation, changes of variables must be introduced that make the $R_{j}$ independent of
the variable $\theta $. It is already known that the average method performs the situation described above, therefore the method will be applied, but with the variant that it will be applied inductively until averaging all the functions $R_{j}$, for all $j=1,\ldots,k,$. Truncating there
the equation and the amplitude of the periodic solutions are obtained approximately. To establish the inductive application of the average method, there is no loss of generality if it is assumed that the coefficients of $\mu ^{j},$ $j=1,\ldots,k-1$, are independent of $\theta $, for this reason, we begin by considering the equations:
\begin{eqnarray}
r^{\prime } &=&\mu R_{1}(r,\alpha )+\ldots +\mu ^{k-1}R_{k-1}(r,\alpha )+\mu
^{k}R_{k}(r,\theta ,\alpha )+O\left( \mu ^{k+1}\right) \\
\theta ^{\prime } &=&w+\mu \omega _{2}(r,\alpha )+\ldots +\mu
^{k-1}w_{k-1}(r,\alpha )+\mu ^{k}W_{k}(r,\theta ,\alpha )+O\left( \mu
^{k+1}\right).
\end{eqnarray}
\label{5.13}
Now, introducing in (\ref{5.13}) the new variables
\begin{equation}
\tilde{r}=r+\mu ^{k}u(r,B,\alpha ),\tilde{\theta}=\theta +\mu ^{k}v(r,\theta
,\alpha )  \label{5.14}
\end{equation}%
with differentials,
\begin{equation*}
\begin{array}{l}
\tilde{r}^{\prime }=r^{\prime }+\mu ^{k}\left( \frac{\partial u}{\partial r}%
(r,\theta ,\alpha )r^{\prime }+\frac{\partial u}{\partial \theta }(r,\theta
,\alpha )\theta ^{\prime }\right) \\
\tilde{\theta}^{\prime }=\theta ^{\prime }+\mu ^{k}\left( \frac{\partial v}{%
\partial r}(r,\theta ,\alpha )r^{\prime }+\frac{\partial v}{\partial
_{\theta }}(r,\theta ,\alpha )\theta ^{\prime }\right),%
\end{array}
\end{equation*}
doing the corresponding substitutions the following equations are obtained:
\begin{equation}
\begin{array}{l}
\tilde{r}^{\prime }=\mu R_{2}(\tilde{r},\alpha )+\ldots +\mu
^{k-1}R_{k-1}\left( \tilde{r},\alpha ^{\prime }\right) +\mu ^{k}\tilde{R}
_{k}\left( p^{\ast },\theta ,\alpha \right) +O\left( \mu ^{k+1}\right) \\
\tilde{\theta}^{\prime }=\omega +\mu \omega _{2}(\tilde{r},\alpha )+\ldots
+\mu ^{k-1}W_{k-1}(\tilde{r},\alpha )+\mu ^{k}\tilde{W}_{k}(\tilde{r},\theta
,\alpha )+O\left( \mu ^{k+1}\right) ,%
\end{array}
\label{5.15}
\end{equation}%
where,
\begin{equation}
\begin{array}{l}
\tilde{R}_{k}(\tilde{r},\tilde{\theta},\alpha )=R_{k}(\tilde{r},\tilde{\theta
},\alpha )+w\frac{\partial u}{\partial \theta }(\tilde{r},\tilde{\theta}
,\alpha ) \\
\tilde{W}_{k}(\tilde{r},\tilde{\theta},\alpha )=W_{k}(\tilde{r},\tilde{\theta
},\alpha )+w\frac{\partial v}{\partial \theta }(\tilde{r},\gamma,\alpha).
\end{array}
\label{5.16}
\end{equation}
The following lemma establishes how the functions $u$ and $v$ should be selected so that $\tilde{R}_{k}$ and $\tilde{W}_{k}$ become independent of the variable $\theta $.
\begin{lemma}
Let be the relation
\begin{equation}
\tilde{A}(r,\theta ,\alpha )=A(r,\theta ,\alpha )+w\frac{\partial b}{%
\partial \theta }(r,\theta ,\alpha ) \label{5.17}
\end{equation}
where $A$ is a given function and $2w$-periodic with respect to $\theta $. If $b$ is chosen as
\begin{equation*}
b(r,\theta ,\alpha )=-(1/w)\int_{0}^{\theta }A(r,s,\alpha )ds+(\theta /2\pi
w)\int_{0}^{2\pi }A(r,s,\alpha )ds
\end{equation*}
then the function, $\tilde{A}(r,\alpha )=\tilde{A}(r,\theta ,\alpha )$, that is, $\tilde{A}$ becomes independent of $\theta $. More specifically, the function $b$ thus selected is $2w$-periodic and the function $\tilde{A}$ corresponding to such function $b$ coincides with the average of $A$, that is,
\begin{equation*}
\tilde{A}(r,\alpha )=(1/2\pi )\int_{0}^{2\pi }A(r,\theta ,\alpha )d\theta .
\end{equation*}
\end{lemma}
\begin{proof}
Note that $b$ es $2\pi$- periodic respect to $\theta$:
\begin{equation*}
\begin{aligned} b(r, \theta+2 \pi, \alpha)=&-(1 / w) \int_{0}^{\theta+2 \pi}
A(r, s, \alpha) d s+\{(\theta+2 \pi) / 2 \pi w) 3 \int_{0}^{2 \pi} A(r, s,
\alpha) d s \\ =&-(1 / w) \int_{0}^{\theta} A(r, s, \alpha) d s-(1 / w)
\int_{\theta}^{\theta+2 \pi} A(r, s, \alpha) d s \\ &+(\theta / 2 \pi w)
\int_{0}^{2 \pi} A(r, s, \alpha) d s+(1 / w) \int_{0}^{2 \pi} A(r, s,
\alpha) d s \\ =&-(1 / w) \int_{0}^{\theta} A(r, s, \alpha) d s+(\theta / 2
\pi w) \int_{0}^{2 \pi} A(r, s, \alpha) d s \\ =& b(r, \theta, \alpha).
\end{aligned}
\end{equation*}
The relation, $\tilde{A}(r,\theta ,\alpha )=(1/2\pi )\int_{0}^{2\pi
}A(r,s,\alpha )ds=\tilde{A}(r,\alpha )$ is easily obtained by substituting $b(r,\theta ,\alpha )$ in the relation (\ref{5.17}).
\end{proof}
It can be observed that to obtain the average in (\ref{5.12}) at the same time the coefficients of $\mu,\;\mu ^{2},\ldots,\;\mu ^{k},$ the succession of average transformations to be applied, can be replaced by a single change of variables
\begin{equation*}
\begin{array}{l}
\tilde{r}=r+\mu u_{2}(r,\theta ,\alpha )+\ldots +\mu ^{k}u_{k}(r,\theta
,\alpha ) \\
\tilde{\theta}=\theta +\mu v_{1}(r,\theta ,\alpha )+\ldots +\mu
^{k}v_{k}(r,\theta ,\alpha )%
\end{array}%
\end{equation*}%
where, the functions $u_{j},v_{j}$ are chosen using the preceding lemma. The above statements will be applied to a concrete problem where the Hopf bifurcation appears. In what follows, all the hypotheses established in Section 3 concerning the autonomous and parameter-dependent differential equation, $x^{\prime }=f(x,\alpha )$, are considered valid. From this equation, it is interesting to use its equivalent form (normal form) (\ref{3.5}):
\begin{equation*}
\begin{array}{l}
x_{1}^{\prime }=\beta (\alpha )x_{1}-\gamma (\alpha )x=+p\left(
x_{1},x_{2},\alpha \right) \\
x_{2}^{\prime }=\gamma (\alpha )x_{2}+\beta (\alpha )x_{2}+q\left(
x_{1},x_{2},\alpha \right) .
\end{array}
\end{equation*}
From the hypotheses (\ref{3.3}) it should be recalled that $\beta ^{\prime }(0)>0$ and according to the implicit function theorem, $\beta (\alpha )$ can be used without loss of generality
as a bifurcation parameter. But for convenience, we will continue to
denote $\alpha$ by $\beta (\alpha )$. Under these assumptions, the last two equations are
rewritten in the form
\begin{equation*}
\begin{array}{l}
x_{1}^{\prime }=\alpha x_{1}-\gamma (\alpha )x_{2}+p\left(
x_{1},x_{2},\alpha \right) \\
x_{2}^{\prime }=\gamma (\alpha )x_{1}+\alpha x_{2}+q\left(
x_{1},x_{2},\alpha \right) .%
\end{array}%
\end{equation*}%
Suppose that $p$ and $q$ admit a Taylor series representation (for this purpose it is sufficient that it is finite with
remainder) such that the above equations can be represented as follows
\begin{equation*}
\begin{array}{l}
x_{1}^{\prime }=\alpha x_{2}-\gamma (\alpha )x_{2}+\Sigma B_{j}^{i}\left(
x_{1},x_{2},\alpha \right) \\
x_{2}^{\prime }=\gamma (\alpha )x_{1}+\alpha x_{2}+\Sigma B_{j}^{i}\left(
x_{2},x_{2},\alpha \right) ,%
\end{array}%
\end{equation*}%
where the $B$ represent polynomials homogeneous of degree $j$ in the variables $x_{1}$ and $x_{2}$, in other words,
for each $j$ the polynomial $B_{j}^{i}$ brings together all the terms of degree $j$ that are obtained in the respective Taylor series
expansions of the functions $p$ and $q$ with respect to the variables $x_{1}$ and $x_{2}$. Now setting $x_{1}=r\cos\theta, \ x_{2}=r\sin \theta $,
to obtain the above equations in polar coordinates:
\begin{equation}
\begin{array}{l}
r^{\prime }=\alpha r+r^{2}C_{2}(\theta ,\alpha )+r^{3}C_{4}(\theta ,\alpha
)+\ldots \\
\theta ^{\prime }=\gamma (\alpha )+rD_{3}(\theta ,\alpha )+r^{2}D_{4}(\theta
,\alpha )+\ldots,
\end{array}
\label{5.18}
\end{equation}
where,
\begin{equation*}
C_{j}(\theta ,\alpha )=(\cos \theta )B_{j-1}^{1}(\cos \theta ,\sin%
\theta ,\alpha )+(\sin\theta )B_{j-1}^{2}(\cos \theta ,\sin%
\theta ,\alpha ),
\end{equation*}
\begin{equation*}
D_{j}(\theta ,\alpha )=(\cos \theta )B_{j-1}^{2}(\cos \theta ,\sin%
\theta ,\alpha )+(\sin\theta )B_{j-1}^{1}(\cos \theta ,\sin%
\theta ,\alpha ).
\end{equation*}
It can be observed that $C_{j}$ and $D_{j}$ are homogeneous polynomials of degree $j$ in $(\cos \theta ,$ sin $\theta )$. We are interested in the periodic solutions of (\ref{5.18}) with the property that
when $\alpha $ tends to zero, then $r$ tends to zero. In this sense, the following change of scale is established
\begin{equation*}
r=\mu \rho ,\quad \alpha =\mu \tilde{\alpha}
\end{equation*}%
in the equation (\ref{5.18}) to obtain
\begin{equation}
\begin{array}{l}
\rho ^{\prime }=\mu \left[ \tilde{\alpha}\varphi +\rho ^{2}C_{3}(\theta ,\mu
\tilde{\alpha})\right] +\mu ^{2}\rho ^{3}C_{4}(\theta ,\mu \tilde{\alpha})+\ldots \\
\theta ^{\prime }=\gamma _{0}+\mu \left[ \tilde{\alpha}\gamma ^{\prime}(0)+\rho D_{3}(\theta ,\mu \tilde{\alpha})\right] +\ldots
\end{array}
\label{5.19}
\end{equation}
Here $\gamma_{0}$ is denoted by the value of $\gamma_(0)$. The equation (\ref{5.19}) does not have exactly the form of the equation (\ref{5.12}), however, it can be seen that the average method can be perfectly applied to the coefficients of $\mu,\mu ^{2},\ldots,$ of the equation corresponding to $\rho ^{\prime}$. Only the coefficients indicated above have to be averaged since the equation of $\theta$ does not provide information on the amplitude of the periodic solution. The generic situation is completely determined by averaging the coefficients of $\mu$ and $\mu^{2}$ of the equation of $r$. To perform the averages, the change of variables is now introduced.
\begin{equation}
\tilde{\rho}=\rho +\mu u_{1}(\rho ,\theta ,\tilde{\alpha},\mu )+\mu
^{2}u_{2}(\rho ,\theta ,\tilde{\alpha},\mu ),  \label{5.20}
\end{equation}
the appearance of $\mu$ in the arguments of $u_{1}$ and $u_{2}$ is due to the fact that $\mu$ now appears in the coefficients of (\ref{5.19}). It is also observed that
\begin{equation*}
\rho =\tilde{\rho}-\mu u_{1}(\tilde{\rho},\theta ,\tilde{\alpha},\mu
)+O\left( \mu ^{2}\right)
\end{equation*}
in agreement with what was demonstrated in Section 4. Now (\ref{5.20}) is substituted in (\ref{5.19}), for convenience the arguments of the functions $u_{1}$ and $u_{2}$ are not included, we obtain:
\begin{eqnarray}
\tilde{\rho}^{\prime } &=&\rho ^{\prime }+\mu \frac{\partial u_{2}}{\partial
\rho }\rho ^{\prime }+\mu \frac{\partial u_{x}}{\partial \theta }\theta
^{\prime }  \notag \\
&=&\mu \left[ \tilde{\alpha}\rho ^{+}\rho ^{2}C_{3}(\theta ,\mu \tilde{\alpha%
})\right] +\mu ^{2}\rho ^{3}C_{4}(\theta ,\mu \widetilde{\alpha })+\mu ^{2}%
\frac{\partial u_{1}}{\partial \rho }\left[ \tilde{\alpha}\rho ^{+}\rho
^{2}C_{3}(\theta ,\mu \tilde{\alpha})\right]  \notag \\
&\qquad +&\mu \frac{\partial u_{2}}{\partial \theta }\left[ \gamma _{0}+\mu
\tilde{\alpha}\gamma ^{\prime }(O)+\mu \rho D_{3}(\theta ,\mu \tilde{\alpha})%
\right] +\mu ^{2}\frac{\partial u_{z}}{\partial \theta }\gamma _{0}+O\left(
\mu ^{3}\right)  \notag \\
&=&\mu \left[ \tilde{\alpha}\rho +\rho ^{2}C_{S}(\theta ,\mu \tilde{\alpha})+%
\frac{\partial u_{2}}{\partial \theta }\gamma _{\theta }\right]  \notag \\
&\qquad +&\mu ^{2}\left\{ \rho ^{3}C_{4}(\theta ,\mu \tilde{\alpha})+\frac{%
\partial u_{2}}{\partial \rho }\left[ \tilde{\alpha}\rho +\rho
^{2}C_{3}(\theta ,\mu \tilde{\alpha})\right] \right.  \notag \\
&\qquad +&\frac{\partial u_{1}}{\partial \theta }\left[ \tilde{\alpha}\gamma
\cdot (0)+\rho D_{s}(\theta ,\mu \tilde{\alpha})\right] +\gamma _{0}\frac{%
\partial u_{2}}{\partial \theta }+O\left( \mu ^{3}\right)  \notag \\
&=&\mu \left[ \tilde{\alpha}\tilde{\rho}+\tilde{\rho}^{2}C_{3}(\theta ,\mu
\tilde{\alpha})+\frac{\partial u_{x}}{\partial \theta }\gamma _{0}\right]
+\mu ^{2}\left\{ \tilde{\rho}^{3}C_{4}(\theta ,\mu \tilde{\alpha})\right.
\notag \\
&\qquad +&\frac{\partial u_{1}}{\partial \rho }\left[ \tilde{\alpha}\tilde{%
\rho}+\tilde{\rho}^{2}C_{3}(\theta ,\mu \tilde{\alpha})\right] +\frac{%
\partial u_{1}}{\partial \theta }\left[ \tilde{\alpha}\gamma \cdot (0)+%
\tilde{\rho}D_{3}(\theta ,\mu \tilde{\alpha})\right]  \notag \\
&\qquad -&u_{1}\left[ \tilde{\alpha}+2\tilde{\rho}C_{s}(\theta ,\mu \tilde{%
\alpha})+\frac{\partial ^{2}u_{x}}{\partial \rho \partial \theta }\gamma
_{0}J+\frac{\partial u_{2}}{\partial \theta }\gamma _{0}{\ }^{3}+O\left( \mu
^{3}\right) .\right.  \notag
\end{eqnarray}%
Following the previous lemma, it can be seen that the function $u_{1}$ is given by
\begin{equation*}
u_{1}(\tilde{\rho},\theta ,\alpha ,\mu )=-\left( \tilde{\rho}^{2}/\gamma
_{0}\right) \int_{0}^{\theta }C_{3}(s,\mu \tilde{\alpha})ds.
\end{equation*}
The other integral that should appear in $u_{1}$ is equal to zero since $C_{3} $ is an homogeneous
trigonometric polynomial of the third degree and therefore has zero average. The coefficient of $p$
in the averaged equation is therefore $\mu \tilde{\rho}$. Now we will search for
$u_{2}$, in this case, it is not necessary to explicitly use the formula of the previous lemma because $u_{1}$
is known and it is also known from the average method that the coefficient that will accompany
$\mu^{2}$ in the averaged equation is the average of the coefficient that accompanies $\mu^{2}$
in the current equation of $\tilde{\rho}$, therefore, the coefficient in question will be
\begin{eqnarray}
R_{2}(\tilde{\rho},\alpha ,\mu ) &=&\mathit{\ average}\text{ }\mathit{of}%
\quad \{\tilde{\rho}^{3}C_{4}+\frac{\partial u_{1}}{\partial \rho }\left[
\tilde{\alpha}\tilde{\rho}+\tilde{\rho}^{2}C_{3}\right]  \notag \\
&\qquad +&\frac{\partial u_{1}}{\partial \theta }\left[ \tilde{\alpha}\gamma
^{\prime }(0)+\tilde{\rho}D_{3}\right] -u_{1}\left[ \tilde{\alpha}+2\tilde{%
\rho}C_{3}+\frac{\partial^{2}u_{1}}{\partial \rho \partial \theta }\gamma
_{0}\right] \}  \notag \\
&=&\mathit{average}\text{ }\mathit{of}\quad \{\tilde{\rho}^{3}C_{4}-\frac{%
\tilde{\rho}^{3}}{\gamma _{0}}C_{3}D_{3}\}  \notag \\
&=&\tilde{\rho}^{3}K.  \notag
\end{eqnarray}
To fix the ideas of what has been exposed in all the above, it is done by summarizing everything said before in the following theorem:
\begin{theorem}
Consider the differential equation
\begin{equation}
\begin {aligned}
\rho^{\prime}&=\mu\left[\tilde{\alpha} \rho+\rho^{2}
C_{3}(\theta, \mu \tilde{\alpha})\right]+\mu^{2} \rho^{-s} C_{4}(\theta, \mu
\tilde{\alpha})+ O \left(\mu^{3}\right)\\
\theta^{\prime}&=\gamma_{0}+\mu\left[\tilde{\alpha} \gamma^{\prime}(0)+\rho
D_{3}(\theta, \mu \tilde{\alpha})\right]+O\left(\mu^{2}\right)
\end{aligned}
\label{5.21}
\end{equation}
which comes from a Hopf bifurcation problem in $\mathbb{R}^{2}$ brought to polar coordinates $(r,\theta )$
and subsequently rescaled by $r=\mu \rho ,\alpha =\mu \tilde{\alpha}$. Then, the change of variables
\begin{equation*}
\tilde{\rho}=\rho +\mu u_{1}(\rho ,\theta ,\tilde{\alpha},\mu )+\mu
^{2}u_{2}(\rho ,\theta ,\tilde{\alpha},\mu ),
\end{equation*}
brings equation (\ref{5.21}) to the averaged form
\begin{equation}
\begin{aligned} \tilde{\rho}^{\prime}&=\mu \tilde{\alpha}
\tilde{\rho}+\mu^{2} \tilde{\rho}^{-s} K+D\left(\mu^{3}\right)\\
\theta^{\prime}&=\gamma_{0}+\mu\left[\tilde{\alpha}
\gamma^{\prime}(0)+\tilde{\rho}^{D} \bar{s}(\theta, \mu
\tilde{\alpha})\right]+D\left(\mu^{2}\right) \end{aligned}  \label{5.22}
\end{equation}%
where $K$ is the constant
\begin{equation}
K=(1/2\pi )\int_{0}^{2\pi }\left\{ C_{4}(\theta ,0)-\left( 1/\gamma
_{0}\right) C_{3}(\theta ,0)D_{3}(\theta ,0)\right\} d\theta .  \label{5.23}
\end{equation}
\end{theorem}
The theorem that follows analyzes what happens in the generic situation of the procedure described above and is summarized in the previous theorem, which occurs for $K\neq 0$ (of the non-generic case, when $K=0$, details will not be given in this work, but in this situation, then more terms of (\ref{5.21}) should be averaged).
\begin{theorem}
Suppose the constant $K$ defined in (\ref{5.23}) is negative, $\mu =\tilde{\alpha}$ and $\tilde{\rho}$ is close to the value
\begin{equation*}
\tilde{\rho}_{0}=(-K)^{-1/2}.
\end{equation*}
Then all periodic solutions obtained by Hopf bifurcation in the original problem
\begin{equation*}
x^{\prime }=f(x,\alpha ), \ x\in\mathbb{R}^{2}
\end{equation*}
that branch from $r=0,\alpha =0,$ are preserved after rescaling $r=\mu \rho ,\alpha =\mu \tilde{\alpha}$ and averaging leading to equation (\ref{5.22}).
Moreover, around each such solution, the positively invariant annular region can be formed
\begin{equation}
A=\left\{ \tilde{\rho}:(1-\varepsilon )\tilde{\rho}_{0}<\tilde{\rho}%
<(1+\varepsilon )\tilde{\rho}_{\theta }\right\} \label{5.24}
\end{equation}
where, $\varepsilon =\varepsilon (\mu )$ is chosen such that, $\delta\longrightarrow 0$ as $\mu \longrightarrow 0$. If $K$ is positive, similar results will be obtained by taking
$\mu =-\tilde{\alpha}\quad $ and $\tilde{\rho}$
now close to $\tilde{\rho}_{0}=K^{-1/2}$, in this case the annular region A will be negatively invariant.
\end{theorem}
\begin{proof}
Since $\mu =\tilde{\alpha}$, in terms of $\mu ,$equation (\ref{5.22}) takes the form
\begin{eqnarray}
\tilde{\rho}^{\prime } &=&\mu ^{2}\left( \tilde{\rho}+\tilde{\rho}^{3}K\right) +O(\mu ^{3}) \notag \\
\theta ^{\prime } &=&\gamma _{0}+O(\mu ^{2}), \notag
\end{eqnarray}
According to Theorem 8 (of integral averages) and the equation of $\tilde{\rho}$ it follows that there exists a (unique) periodic solution in a
neighborhood of $\tilde{\rho}_{0}=(-K)^{-1/2}$. To see that the periodic solutions of the original problem are preserved after rescaling, consider any periodic solution of (\ref{5.18}) (which still preserves the initial scale) that branches off from $r=0,\alpha =0$, then for some point $(r_{1},\theta _{1})$ of the solution in question $r^{\prime }$ must vanish, due to periodicity. After rescaling by $\mu=r_{1}/\tilde{\rho}_{0}$, $r^{\prime }$ must vanish at the point with coordinates $(\tilde{\rho}_{0},\theta _{1})$. Therefore,
\begin{equation*}
0=\mu \alpha \tilde{\rho}_{0}+\mu ^{2}\rho ^{3}K+O\left( \mu ^{3}\right)
=\mu \tilde{\rho}_{0}(\tilde{\alpha}-\mu +O\left( \mu ^{2}\right) ).
\end{equation*}
For the second equality, the value of $\tilde{\rho}_{0}=(-K)^{-1/2}$ has been taken into account, \ and it is concluded from the last equality that
\begin{equation*}
\tilde{\alpha}=\mu +O\left( \mu ^{2}\right)
\end{equation*}
and for this reason, for $\mu =\tilde{\alpha},$ all the original periodic solutions prevail in a neighborhood of $\tilde{\rho}_{0}$.
Consider now the annular region defined by (\ref{5.24}) , from the
following conditions, it can be seen that it is possible to appropriately select a function $\varepsilon =\varepsilon (\mu )$ such that $A$ is positively invariant:
i) For $\tilde{\rho}=(1+\delta )\tilde{\rho}_{0}$, we have that
\begin{equation*}
\tilde{\rho}^{\prime }=\mu ^{2}(1+\delta )\tilde{\rho}_{0}[(\tilde{\alpha}/\mu )-(1+\delta )^{2}+O(\mu )]<0.
\end{equation*}
ii) For $\tilde{\rho}=(1-\delta )\tilde{\rho}_{0},$we have that
\begin{equation*}
\tilde{\rho}^{\prime }=\mu ^{2}(1-\delta )\tilde{\rho}_{0}\left[ (\tilde{\alpha}/\mu )-(1-\delta )^{2}+O(\mu )\right] <0.
\end{equation*}
To see the inequality, it must be noted that, $\tilde{\alpha}/\mu=1+O(\mu )$. From i) and ii) it is also inferred that the periodic solution is located entirely in the annular region $A$.
\end{proof}
\section{Application of the Method}
The results of the previous sections will be applied to the two-predator, one-prey ecological model given in
\cite{HHW}. A detailed study of this model can be found in \cite{CavaniMaster} and a generalization of this model can be found in
\cite{CavaniRD}. So, the model is expressed by the following system of ordinary differential equations:
\begin{eqnarray}
S^{\prime } &=&\gamma S\left( 1-\frac{S}{k}\right) -\frac{m_{1}x_{1}S}{a_{1}+S}-\frac{m_{2}x_{2}S}{a_{2}+S} \notag \\
x_{1}^{\prime } &=&\frac{m_{1}x_{1}S}{a_{1}+S}-d_{1}x_{1} \label{6.1} \\
x_{2}^{\prime } &=&\frac{m_{2}x_{2}S}{a_{2}+S}-d_{2}x_{2}, \notag
\end{eqnarray}
where, $S$ represents the density of the prey population, $x_{1}$ and $x_{2}$ represent the densities of the predator populations; as can be seen,
the prey population shows a logistic growth in the absence of predators and the trophic function or performance response follows the well-known Michaelis-Menten kinetic;
$\gamma >0$ is the intrinsic growth rate of the prey, $k>0$ is the environmental saturation constant of the prey; $m_{i}>0,$ $d_{i}>0,$ $a_{i}>0$ are the maximum birth rate,
the death rate, and the "semi-saturation constant" of the $i$-th predator respectively for $i=1.2$; the apostrophe indicates derivation for time $t$.
The meaning of the semi-saturation constant is that for $S=a_{i}$
the performance response of the $i$-th predator is equal to $m_{i}/2$, half the maximum birth rate.
It can be seen that the equation for the $i$-th predator vanishes when the pair $\left( S,x_{1}\right) =(0,0)$
or else when $S=a_{i}d_{i}/\left( m_{i}-d_{i}\right) ,\quad i=1,2. $ Define,
\begin{equation*}
\lambda _{i}=\frac{a_{i}d_{i}}{m_{i}-d_{i}},\quad i=1,2.
\end{equation*}
In \cite{HHW} it is shown that the solutions of (\ref{6.1}) corresponding to positive initial conditions are bounded and remain in the first octant, also
they showed that a necessary condition for the survival of the $i$-th predator is $0<\lambda _{i}<k$. The authors mentioned dealt
mainly with the generic case of the model, that is, when $\lambda _{1}\neq \lambda _{2}$.
The non-generic case $\lambda _{1}=\lambda _{2}$, which can be interpreted as predators being equally voracious, has been studied in
\cite{HalSmith} and in \cite{MFarkas}. In the present work, the non-generic case $\lambda _{1}=\lambda _{2,}$ is studied but under the condition
non-generic $a_{1}=a_{2}$. Specifically, we are interested in periodic solutions that come from Hopf bifurcation to apply the procedures of Section 5.
The following notations will facilitate the study of the problem:
\begin{itemize}
\item N. 1) It is assumed that $\lambda =\lambda _{1}=\lambda _{2,}\quad
a=a_{1}=a_{2}.$
\item N. 2) We denote $b_{i}=m_{i}/d_{i},i=1,2$.
This notation allows us to write $\lambda =a/\left( b_{i}-1\right) ,$ from this, it can be seen that, $b_{1}=b_{2}$. Therefore, from now on, we will assume that
\begin{equation*}
b=b_{1}=b_{2}.
\end{equation*}
\item N.3) From the equality $b_{1}=b_{2},$ the following follows: $%
d_{1}/d_{2}=m_{1}/m_{2}.$ The common quantity
is denoted by $\rho =d_{1}/d_{2}=m_{1}/m_{2}=$. With this notation can not be put
\begin{equation*}
m_{2}=\rho \;m_{1}.
\end{equation*}
\item N. 4) It is denoted by $\beta _{i}=m_{i}-d_{i},$ $i=1,2.$ It is observed that,
\begin{equation*}
m_{1}-d_{1}=d_{1}\left[ \left( m_{1}/d_{1}\right) -1\right] =d_{1}\left(
b_{1}-1\right)
\end{equation*}
for this reason, too $\rho =\beta _{2}/\beta _{1}$ and $\beta _{2}=\rho \;\beta 1.$
\end{itemize}
Hereinafter it is also assumed that the following hypothesis is fulfilled:
\begin{equation*}
\text{\textbf{\ H)}\qquad \qquad }b>1, 0<\lambda <k, \text{and}\rho \geq 1.
\end{equation*}
With the above notations, the system (\ref{6.1}) becomes the following:
\begin{eqnarray}
S^{\prime } &=&\gamma S\left( 1-\frac{S}{k}\right) -\left( x_{1}+\rho
x_{2}\right) \frac{m_{1}s}{a+s} \notag \\
x_{1}^{\prime } &=&\beta _{1}x _{1}\frac{(s-\lambda )}{a+s} \label{6.2}
\\
x_{2}^{\prime } &=&\rho \beta _{1}x_{2}\frac{(s-\lambda )}{a+s}. \notag
\end{eqnarray}
The equilibrium points of the system (\ref{6.2}) are: $(0,0,0)$, $(k,0,0)$ and
all points that lie on the straight-line segment
\begin{equation*}
L=\{(S,x_{1}x_{2})\in \mathbb{R}^{3}:S=\lambda ,x_{1}+\rho x_{2}=\gamma (a+\lambda )(k-\lambda
)/m_{1}k,x_{1},x_{2}\geq 0\}
\end{equation*}
Using the linearization analysis of the system at the points $(0,0,0)$ and $(k,0,0)$ it is easily demonstrated
that the equilibrium points are unstable (the calculations to be performed are practically the same as those in
\cite{CavaniMaster}.
For the analysis of the stability of the points of the segment $L$, the situation is somewhat more complicated,
since it is a continuum of equilibrium points. According to the procedure in \cite{HalSmith}, the parameter $t$
is eliminated from the last two equations of (\ref{6.2}) and gives,
\begin{equation*}
\frac{dx_{2}}{dx_{1}}=\rho \frac{x_{2}}{x_{1}}
\end{equation*}%
This implies that, $Ln\left( x_{2}\right) = Ln\left( x_{1}^{\rho
}\right) + Ln(c),c\geq 0$. Then,
\begin{equation}
x_{2}=cx_{1}^{\rho }. \label{6.3}
\end{equation}
Now, sketch the system (\ref{6.2}) restricted to the invariant variety (\ref{6.3}) and parameterized by $S$ and ${x}_{1}$, i.e.
\begin{equation}
\begin{aligned} S^{\prime}&=\gamma
S\left(1-\frac{S}{k}\right)-\left(x_{1}+\rho c x_1^{\rho}\right) \frac{m_{1}
S}{a+S}\\ x_{1}^{\prime}&=\beta_{1} x_{1} \frac{S-\lambda}{a+S}.
\end{aligned} \label{6.4}
\end{equation}
The equilibrium points of (\ref{6.4}) are: $(0,0),(k,0)$ and the point $Q$
of the line segment $L$ in common with the variety (\ref{6.3}). To be more
specific, regarding the point $Q$, note the following: if we denote by $Q^{\ast }=\left( \lambda ,\xi _{1},\xi
_{2}\right) ,$ the intersection point of $L$ with the manifold (\ref{6.3}), where,
\begin{equation*}
\xi _{1}+\rho \xi _{2}=\gamma (a+\lambda )(k-\lambda )/m_{1}k\quad and\quad
\xi _{2}=c\xi _{1}^{\rho }
\end{equation*}
then the coordinates of $Q$ are $\left( \lambda ,\xi _{1}\right)$, where
$\xi _{1}$ satisfies the following relation:
\begin{equation}
\xi _{1}+\rho c\xi _{1}^{\rho }=\gamma (a+\lambda )(k-\lambda )/m_{1}k.
\label{6.5}
\end{equation}
The second coordinate of $Q$ is a function of $c$, as well as the other
parameters that appear in the relation (\ref{6.5}). In what follows all the parameters appearing in the equation
(\ref{6.4}) will be considered fixed except $k$, which plays the role of a bifurcation parameter.
According to this $\xi _{1}\equiv \xi _{1}(c,k)$. In M. Farkas (1984) the following basic facts are
proven:\newline
\textbf{F.1)} The points$(0,0)$ and $(k,0)$ are unstable\newline
\textbf{F.2)} If $\lambda <k<a+2\lambda $, the equilibrium point $\left(\lambda ,\xi _{1}\right) $
is asymptotically stable.\newline
\textbf{F.3)} If $k>a+2\lambda ,$ the equilibrium point $\left( \lambda
,\xi _{1}\right) $ is unstable.\newline
\textbf{F.4)} For $k=a+2\lambda $, the system (\ref{6.4}) presents the hypotheses of the Bifurcation Theorem
of Hopf.\newline
According to \textbf{F.4} for $k=a+2\lambda ,$ the equilibrium point $%
\left( \lambda ,\xi _{1}\right) $ branches into a periodic orbit. Now we can apply the results of Section 5. We will obtain a certain region of parameters in which the
periodic orbit obtained from the Hopf Theorem is asymptotically orbitally stable. The region of parameters that we will obtain
is more restricted than that obtained by M. Farkas (1984), but it is improved in the sense that we can now establish a concrete region in which the
orbit is located. All the results follow as a consequence of the positive invariance of the ring region $A$ referred to in Theorem 5.3.
To begin with, in the equation (\ref{6.4}), consider
The following change of variables:
\begin{equation*}
y_{1}=S-\lambda ,\quad y_{2}=x_{1}-\xi _{1}.
\end{equation*}%
We obtain,
\begin{equation}
\begin{aligned}
y_{1}^{\prime}&=\gamma\left(y_{1}+\lambda\right)\left(1-\frac{y_{1}+%
\lambda}{k}\right)\\
&-\frac{m_{1}\left(y_{1}+\lambda\right)}{a+\lambda+y_{1}}
\left[y_{2}+\xi_{1}+\frac{\gamma(a+\lambda)(1-\lambda / k)-m_{1}
\xi_{1}}{m_{1}}\left(y_{2}+\xi_{1}\right)^{\rho}\right]\\
y_{2}^{\prime}&=\beta_{1}\left(y_{2}+\xi_{1}\right)
\frac{y_{1}}{a+\lambda+y_{1}}, \end{aligned}  \label{6.6}
\end{equation}
to obtain the equation for $y_{1}$, was substituted $\rho c$, using the equality (\ref{6.5}).
Now, getting the right hand sides of (\ref{6.6}) in Taylor series,
around of $\left( y_{1},y_{2}\right) =(0,0)$ and for $k=a+2\lambda,$ the following equations are obtained:
\begin{equation}
\begin{aligned} y_{1}^{\prime}=&-\frac{\lambda
m_{1}}{a+\lambda}\left(1-\rho+\frac{\rho \gamma (a+\lambda)^2}{m_{1}
\xi_{1}(a+2 \lambda)}\right) y_2\\ &-\frac{\gamma \lambda}{(a+\lambda)(a+2
\lambda)} y_{1}^{2}\\ &-\frac{a m_1}{(a+\lambda)^{2}}\left(1-\rho+\frac{\rho
\gamma(a+\lambda)^2}{m_{1} \xi_{1}(a+2 \lambda)}\right) y_{1} y_{2}\\
&\left.-\frac{\lambda \rho(\rho-1)}{(a+\lambda) 2 \xi_{1}^{2}}
\frac{\gamma(a+\lambda)^{2}}{a+2 \lambda}-m_{1} \xi_{1}\right) y_{2}^{2}\\
&-\frac{a\gamma}{(a+\lambda)^{2}(a+2 \lambda)} y_{1}^{3}\\ &+\frac{a
m_{2}}{(a+\lambda)^{3}}\left(1-\rho+\frac{\rho_{1}}{m_{1} \xi_1 (a+2
\lambda)}\right) y_{1}^{2} y_2\\ &-\frac{a \rho(\rho-1)}{(a+\lambda)^{2} 2
\xi_{1}^{2}}\left(\frac{r(a+\lambda)^{2}}{a+2 \lambda}-m_{1} \xi_{1}\right)
y_{1} y_{2}^{2}\\ &-\frac{\lambda \rho(\rho-1)(\rho-2)}{b
\xi_{1}^{2}}\left(\frac{\gamma(a+\lambda)^{2}}{a+2 \lambda}-m_{1}
\xi_{1}\right) y_{2}^{3}+\cdots \end{aligned}  \label{6.7a}
\end{equation}
\begin{equation}
\begin{aligned} y_{2}^{\prime}=&-\frac{\beta_{1} \xi_{1}}{a+\lambda} y_{1}\\
&-\frac{\beta_{1} \xi_{1}}{(a+\lambda)^2} y_{1}^{2}\\
&+\frac{\beta_{1}}{a+\lambda} y_{1} y_{2} \\ &+\frac{\beta_{1}
\xi_{1}}{(a+\lambda)^{3}} y_1^{3}-\frac{\beta_{1}}{(a+\lambda)^{2}}
y_{1}^{2} y_{2}+\ldots \end{aligned}  \label{6.7b}
\end{equation}%
Equations (\ref{6.7a}) and (\ref{6.7b}) do not have the canonical Jordan form, but it is known
(from matrix theory in Jordan form) that there exists a transformation of variables $\mathbb {T}$ from
$\mathbb{R}^{2}$ to $\mathbb{R}^{2}$ such that if $col (y_{1},y_{2})=\mathbb{T}(col(u,v))$, then the above equations take the form
\begin{equation}
u^{\prime }=w(a+2\lambda )v+g_{1}(u,v),\quad v^{\prime }=-w(a+2\lambda
)+g_{2}(u,v), \label{6.8}
\end{equation}%
where, $\pm i\omega (a+2\lambda )$, are the eigenvalues ??of the system
linearized from (\ref{6.7a})-\ref{6.7b}), where
\begin{equation*}
w(a+2\lambda )=\left[ \lambda m_{1}\beta _{1}\xi _{1}\left( 1-\rho +\rho
\gamma (a+\lambda )2/m_{1}\xi _{1}(a+2\lambda )\right) \right]
^{1/2}/(a+\lambda ).
\end{equation*}%
To specify about the transformation $\mathbb {T}$, think the system (\ref{6.6}) in the form
\begin{equation}
y^{\prime }=A(k)y+F(y), y\in \mathbb{R}^{2} \label{6.9}
\end{equation}
then the transformation $\mathbb{T}\equiv \mathbb{T}(k)$ is given by:
\begin{equation*}
\mathbb{T}=\left[
\begin{array}{ll}
1 & p(k) \\
0 & q(k)
\end{array}
\right],p(k)=\left(\alpha (k)-a_{1}{\ }_{1}(k)\right)
/w(k),\;q(k)=-a_{21}(k)/w(k)
\end{equation*}
where, $\alpha(k)\pm iw(k)$ are the eigenvalues of  the matrix $2\times 2,$ $A(k)$ and $a_{11}$, $a_{21}$
are two of its elements. In this case, we have $k=a+2\lambda ,$ from (\ref{6.7a})-(\ref{6.7b})
it is observed that $a_{11}=0$ and from {\textbf{F}.4},
\begin{equation*}
\alpha (a+2\lambda )=0\quad {\mathit{and}}\quad w(a+2\lambda )\neq 0,
\end{equation*}
well, they are part of the hypothesis of Hopf's Theorem. Therefore,
\begin{equation*}
p(a+2\lambda )=0\quad \mathit{and}\quad q(a+2\lambda )=-\beta _{1}\xi
_{1}/(a+\lambda )w(a+2\lambda ).
\end{equation*}
Putting: $q=q(a+2\lambda )$ and $w=w(a+2\lambda )$, it can be summarized that
\begin{equation*}
\mathbb{T}=\left[
\begin{array}{ll}
1 & 0 \\
0 &q%
\end{array}%
\right] ,\quad {\mathbb{T}}^{-1}=\left[
\begin{array}{ll}
1 & 0 \\
0 & 1/q%
\end{array}%
\right],
\end{equation*}%
and the change of variables is:
\begin{equation*}
\left[
\begin{array}{c}
y_{1} \\
y2%
\end{array}%
\right] =\mathbb{T}\left[
\begin{array}{l}
or \\
v
\end{array}
\right] =\left[
\begin{array}{ll}
1 & 0 \\
0 &q%
\end{array}%
\right] \left[
\begin{array}{l}
u \\
v
\end{array}%
\right] =\left[
\begin{array}{l}
u \\
qv
\end{array}
\right],
\end{equation*}
that is to say,
\begin{equation}
y_{1}=u,\quad y_{z}=qv. \label{6.10}
\end{equation}
Under the change of variables (\ref{6.10}) the system (\ref{6.7a})-(\ref{6.7b}) takes
the form (\ref{6.8}) and it can be noted that in this equation the linear part has the canonical Jordan form.
To determine the non-linear part of (\ref{6.8}), one must
take into consideration the non-linear terms of (\ref{6.7a})-(\ref{6.7b}) which under the representation (\ref{6.9}) can easily be seen that
\begin{equation*}
\left[
\begin{array}{l}
g_{1}\left(u,v\right) \\
g_{2}(u,v)
\end{array}
\right] ={\mathbb{T}}^{-1}\quad F\left( \mathbb{T}\left[
\begin{array}{l}
u \\
v
\end{array}
\right] \right).
\end{equation*}
This calculation is performed below. For convenience, in (\ref{6.7a})-(\ref{6.7b}), regardless of their signs,
the coefficients of the non-linear part of $y_{1}$ are designated by $\bar{a}_{i}$ and those of $y_{2}$ by $\bar{b}_{i}$.
So, putting $\mathtt{B}=F(\mathbb{T}(col (u,v))),\;\mathtt{C}={\mathbb{T}}^{-1}F(\mathbb{T}(col(u,v)))$ and taking into account (\ref{6.10}), we have to
\begin{equation*}
\mathtt{B}=\left[
\begin{array}{c}
\left. -\bar{a}_{1}u^{2}-\bar{a}_{2}quv-\bar{a}_{3}q^{2}v^{2}-\bar{a}
_{4}u^{3}+\bar{a}_{5}qu^{2}v-\bar{a}_{6}q^{2}uv^{2}-\bar{a}%
_{7}q^{3}v^{3}+\ldots \right] \\
-\bar{b}_{1}u^{2}+\bar{b}_{2}quv-\bar{b}_{3}u^{3}-\bar{b}_{4}u^{2}v+\ldots
\end{array}
\right.
\end{equation*}
and
\begin{equation*}
\mathtt{C}=\left[
\begin{array}{rl}
-\bar{a}_{1}u^{2}-\bar{a}_{2}quv-\bar{a}_{3}q^{2}v^{2}-\bar{a}_{4}u^{3}+\bar{%
a}_{5}qu^{2}v-\bar{a}_{6}q^{2}uv^{2}-\bar{a}_{7}q^{3}v^{3}+\ldots & \\
-\left( \bar{b}_{1}/q\right) u^{2}+\bar{b}_{2}uv-\left( \bar{b}_{3}/q\right)
u^{3}-\left( \bar{b}_{4}/q\right) u^{2}v+\ldots &
\end{array}%
\right].
\end{equation*}%
The vector $\mathtt{C}$ is the nonlinear part of (\ref{6.8}). With the
results of the previous calculations, rewrite (\ref{6.8}) as
\begin{equation}
\begin{aligned} v^{\prime}=-w u-\left(\bar{b}_{1} / q\right)
u^{z}+\bar{b}_{2} u v-\left(\bar{b}_{3} / q\right) u^{3}-\left(\bar{b}_{4} /
q\right) u^{2} v+\ldots \\ v^{\prime}=-w u-\left(\bar{b}_{1} / q\right)
u^{2}+\bar{b}_{2} u v-\left(\bar{b}_{3} / q\right) u^{3}-\left(\bar{b}_{4} /
q\right) u^{2} v+\ldots \end{aligned} \label{6.11}
\end{equation}
With $u=r\;cos\;\theta $ and\ $v=r\;sin$ $\theta $, we pass to an equation
with polar coordinates as (\ref{5.18}). So, the following system is obtained:
\begin{equation}
\begin{aligned} r^{\prime}&=r^{2}\left[-\bar{a}_{1} \cos ^{3}
\theta-\left(\bar{a}_{2} q+\bar{b}_{2} / q\right) \cos ^{2} \theta
\operatorname{sin} \theta+\left(\bar{b}_{2}-\bar{a}_{3}\right)
\operatorname{sin}^{2} \theta \cos \theta\right]\\ &\quad
+r^{3}\left[-\bar{a}_{4} \cos ^{4} \theta+\left(\bar{a}_{5} q+\bar{b}_{3} /
q\right) \cos ^{3} \theta \operatorname{sin} \theta-\left(\bar{a}_{6}
q^{2}+\bar{b}_{4} / q\right) \cos ^{2} \theta \operatorname{sin}^{2}
\theta\right. \\ &\quad \left.-\bar{a}_{7} q^{3} \operatorname{sin}^{3}
\theta \cos \theta\right]+O\left(r^{4}\right)\\
\theta^{\prime}&=-w+r[-\left(\bar{b}_{1} / q\right) \cos ^{3}
\theta-\bar{a}_3q^{2} \operatorname{sin}^{3} \theta-\bar{a}_{2} q \cos
\theta \operatorname{sin}^{2} \theta \\ &\quad
+\left(\bar{b}_2-\bar{a}_{1}\right) \cos ^{2} \theta \operatorname{sin}
\theta ]+O\left(r^{2}\right) \end{aligned}  \label{6.12}
\end{equation}
To apply Theorem 10, the equations (\ref{6.12}) are sufficient, from which
all the data can be collected to calculate the constant $K$ referred to in
said theorem. In this case,
\begin{equation*}
K=(1/2\pi )\int_{0}^{2\pi }\left[ C_{4}(\theta ,a+2\lambda
)+(1/w)C_{3}(\theta ,a+2\lambda )D_{3}(\theta ,a+2\lambda )\right] d\theta
\end{equation*}%
where, it has been taken from (\ref{6.12}) and in accordance with what is
established in Theorem 10: $C_{3}(\theta ,a+2\lambda )$ and $C_{4}(\theta
,a+2\lambda )$ are the respective coefficients of $r^{2}$ and of $r^{3}$ in
the equation for $r$, and $D_{3}(\theta ,a+2\lambda )$, the coefficient of $%
r $ in the equation for $\theta $. By evaluating the integrals that result
from substituting the coefficients indicated above in the expression for $K$%
, we obtain:
\begin{equation*}
\begin{aligned} K=&-(3 / 8) \bar{a}_{4}-(1 / 8) \bar{a}_{6} q^{2}-(1 / 8 q)
\bar{b}_{4}+(1 / w q) \bar{a}_{1} \bar{b}_{1}+(1 / w) \bar{a}_{1}
\bar{a}_{2} q \\ &-(1 / w) \bar{a}_2 \bar{b}_{2} q-(1 / w q) \bar{b}_{1}
\bar{b}_{2}+(1 / 2 w q) \bar{a}_{3} \bar{b}_{1}+(1 / 2 w) \bar{a}_{2}
\bar{a}_{3} q\left(1+q^{2}\right) \end{aligned}
\end{equation*}
Replacing the value $q=-\beta _{1}\xi _{1}/(a+\lambda )w$, in the previous
expression gives
\begin{equation*}
\begin{aligned} 8 w^{2} K=&-3 \bar{a}_{4} w^{2}-\beta_{1} \xi_{1}
\bar{a}_{6} /(a+\lambda)=+(a+\lambda) \bar{b}_{4} w^{3} / \beta_{1} \xi_{1}
\\ &-(a+\lambda) \bar{a}_{1} \bar{b}_{1} w^{2} / \beta_{1} \xi_{1}-\beta_{1}
\xi_{1} \bar{a}_{1} \bar{a}_{2} /(a+\lambda) \\ &+\beta_{1} \xi_{1}
\bar{a}_{2} \bar{b}_{2} /(a+\lambda)+(a+\lambda) \bar{b}_{1} \bar{b}_{2}
w^{2} / \beta_{2} \xi_{2} \\ &-(a+\lambda) \bar{a}_{3} \bar{b}_{1} w^{2} / 2
\beta_{1} \xi_{1}-\beta_{1} \xi_{1} \bar{a}_{2} \bar{a}_{3} / 2(a+\lambda)
\\ &-\beta_{1} \xi_{1} \bar{a}_{2} \bar{a}_{3} / 2(a+\lambda)^3w^2,
\end{aligned}
\end{equation*}%
now, replacing the values of $\bar{a}_{1},\bar{a}_{2},\bar{a}_{3},\bar{a}%
_{4},\bar{a}_{6},\bar{b}_{1},\bar{b}_{2},\bar{b}_{4},w^{2}$, and putting
\begin{equation*}
h_{1}\equiv 1-\rho +\rho ^{\prime }(a+\lambda )z/m_{1}\xi _{1}(a+2\lambda
),h_{2}\equiv \gamma (a+\lambda )z/(a+2\lambda )-m_{2}\xi _{1},
\end{equation*}%
then according with equations (\ref{6.7a})-(\ref{6.7b}), that contain $\bar{a%
}_{2}$, $\bar{a}_{3}$ and $\bar{a}_{6}$, turn out to be positive. So, by (%
\ref{6.5}), the following expression is obtained:
\begin{equation}
\begin{aligned} \left[8 w^{2}(a+\lambda)^4 / h_{1}\right] K=&-4 a \gamma
\lambda m_1 \beta_{1} \xi_{1} /(a+2 \lambda)-2 a \rho\left(\rho -1\right)
\beta_{1} h_{2} / h_{1} \\ &+(a+\lambda) \beta_{1} m_{1} \xi_{1} w-\gamma
\lambda^{2} m_{1} \beta_{1} \xi_{1} /(a+2 \lambda) \\ &+m_{1} \beta_{1}
\xi_{1}\left(a+\lambda_{\beta_{1}}\right) \\ &-(1 / 4) m_{1} \beta_{1}
\xi_{1} \rho (\rho -1)\left(\lambda h_{2}+a\right). \end{aligned}
\label{6.13}
\end{equation}%
As can be seen, $K$ has the same sign as the expression on the right side of (%
\ref{6.13}), and therefore, $K$ will be negative if and only if $%
w<G_{1}/G_{2}$, where
\begin{equation*}
\begin{aligned} G_{1}=& 4 \gamma \lambda m_{2} \beta_{1} \xi_{2}(a+\lambda)
/(a+2 \lambda)+2 a \rho(\rho-1) \beta_{1} h_{z} / h_{1} \\ &+m_{2} \beta_{1}
\xi_{1}\left(a+\lambda \beta_{1}\right) \\ &-(1 / 4) m_{1} \beta_{2}
\rho\left(\rho^{-1)}\left(\lambda h_{z}+a\right)\right. \end{aligned}
\end{equation*}%
and
\begin{equation*}
G_{2}=(a+\lambda )m_{1}\beta _{1}\xi _{1}.
\end{equation*}%
In that case, the $2\pi $-periodic orbit of (\ref{6.12}) (and therefore the
periodic solution of the original problem (\ref{6.4}) is orbitally asymptotically
stable. Furthermore, using Theorem 11 it is possible to construct a
positively invariant annular region in which the orbit is located. All the
previous calculations and results are summarized in the following theorem:
\begin{theorem}
Consider the equations (\ref{6.4}) defined on the variety (\ref{6.3}). For $%
k=a+2\lambda $, the critical point of the system with strictly positive
coordinates, $(\lambda ,\xi _{1})$ bifurcates, according to Hopf's Theorem,
into a periodic orbit that is asymptotically stable orbital if,
\begin{equation*}
w<G_{1}/G_{2}
\end{equation*}%
and unstable if,
\begin{equation*}
w>G_{1}/G_{2}
\end{equation*}%
respectively. The periodic orbit is located in a region homeomorphic to the
annular region given by Theorem 11 taking the value of $K$ obtained by means
of (\ref{6.13}).
\end{theorem}
The following Corollary is obtained as a direct consequence of the previous
theorem.
\begin{corollary}
For $k=a+2\lambda $, $\lambda =\lambda _{1}=\lambda _{2}$ the segment $L$ of
Equilibrium points of (\ref{6.1}) bifurcates into a "cylinder" that is
obtained from the union of the orbits referred to in the previous theorem.
If \ $w<G_{1}/G_{2}$ then the cylinder is an attractor of the system, that
is, there exists a neighborhood of the cylinder such that every trajectory
with initial conditions in this neighborhood tends to the cylinder when t
tends to infinity. Such a neighborhood can be constructed by joining the
annular regions obtained from Theorem 11 considering the respective changes
of variables.
\end{corollary}
\section{Conclusions}
The procedure described in this work allows us to locate an annular region
where the orbit of a periodic solution of a nonlinear ordinary differential
equation is located, specifically, it deals with periodic solutions of
autonomous systems (or also periodic systems) that depend on a parameter. It
has been possible to implement a method that improves results that only
exist when obtaining periodic solutions through the Hopf Bifurcation is
used. The case is fully described for the case of two-dimensional equations.
With the Hopf Bifurcation Theorem, the problem of the existence of periodic
solutions in systems of differential equations that depend on a parameter
can be solved. However, with the method presented here, some characteristics
can be quantitatively known. Properties of the solution such as the amplitude,
the period, or the region where it is located, using the method of integral
averages described here in a general way can be known. The important thing
is to be able to show, through the solution of a non-trivial problem, the
steps necessary to obtain, approximately, the amplitude and stability of a
periodic solution that is obtained through the bifurcation of a critical point
of an autonomous equation. The method is applied to a three-dimensional
system of differential equations that models the competition of two
predators and prey, a model that was proposed by Hsu, Hubbell, and Waltman
in \cite{HHW}, obtaining results that improves previous one in \cite{MFarkas}.

\end{document}